\documentclass[11pt,a4paper]{amsart}
\usepackage[english]{babel}
\usepackage{enumerate}
\usepackage{amsmath,amssymb}
\usepackage[dvips]{graphicx}
\usepackage[size=footnotesize]{caption}
\usepackage[size=footnotesize]{subcaption}

\usepackage{listings}

\usepackage{wasysym}
\usepackage{amssymb,amsfonts, color,epsf}
\usepackage{verbatim}
\usepackage{fullpage}\usepackage{bbm,amsmath,amssymb,amsfonts,graphicx,color,mathtools, verbatim}
   \usepackage[bookmarksopen=false,pdftex=true,breaklinks=true,%
      backref=page,pagebackref=true,plainpages=false,%
      hyperindex=true,pdfstartview=FitH,colorlinks=true,%
      pdfpagelabels=true,colorlinks=true,linkcolor=blue,%
      citecolor=red,urlcolor=green,hypertexnames=false%
      ]%
   {hyperref}
   \newcommand{\abb}[5]{%
\setlength{\arraycolsep}{0.4ex}%
\begin{array}{rcccc}%
#1 &:\,& #2 & \,\,\longrightarrow\,\, & #3 \\[0.5ex]%
     & & #4 & \longmapsto & #5%
\end{array}%
}

\newcommand{\N}{\mathbb{N}}

\newcommand{\R}{\mathbb{R}}
\newcommand{\C}{\mathbb{C}}

\newcommand{\X}{\underline{X}}
\renewcommand{\H}{\mathcal{H}}
\newcommand{\A}{\mathcal{A}}
\newcommand{\B}{\mathcal{B}}
\newcommand{\E}{\mathcal{E}}
\newcommand{\W}{\mathcal{W}}
\renewcommand{\P}{\mathcal{P}}
\renewcommand{\leq}{\leqslant}
\usepackage{faktor}

\usepackage{MnSymbol} 
\newtheorem{theorem}{Theorem}[section]
\newtheorem{lemma}[theorem]{Lemma}
\newtheorem{corollary}[theorem]{Corollary}
\newtheorem{proposition}[theorem]{Proposition}

\newtheorem{definition}[theorem]{Definition}

\newtheorem{remark}[theorem]{Remark}
\newtheorem{example}[theorem]{Example}
\usepackage[top=2cm, bottom=2cm, left=3cm, right=3cm]{geometry}

\DeclarePairedDelimiter\floor{\lfloor}{\rfloor}

\DeclareMathOperator\rank{rank}

\DeclareMathOperator{\Hom}{Hom}

\DeclareMathOperator{\Span}{span}

\DeclareMathOperator{\sDisc}{sDisc}

\DeclareMathOperator{\conv}{conv}

\DeclareMathOperator{\locextr}{locextr}
\DeclareFontFamily{U}{mathx}{\hyphenchar\font45}
\DeclareFontShape{U}{mathx}{m}{n}{
      <5> <6> <7> <8> <9> <10>
      <10.95> <12> <14.4> <17.28> <20.74> <24.88>
      mathx10
      }{}
\DeclareSymbolFont{mathx}{U}{mathx}{m}{n}
\DeclareFontSubstitution{U}{mathx}{m}{n}
\DeclareMathAccent{\widecheck}{0}{mathx}{"71}

\title{Linear slices of Hyperbolic polynomials and positivity of symmetric polynomial functions}
\author{Cordian Riener \and Robin Schabert}
\address{Department of Mathematics and Statitics, UiT - the Arctic University of Norway, 9037 Troms\o, Norway}
\email{cordian.riener@uit.no}
\email{robin.schabert@uit.no}
\date{\today}

\begin{document}
\begin{abstract}
A real univariate polynomial of degree $n$ is called hyperbolic if all of its $n$ roots are on the real line. Such polynomials appear quite naturally in different applications, for example, in combinatorics and optimization. The focus of this article are families of hyperbolic polynomials which are determined through $k$ linear conditions on the coefficients. The coefficients corresponding to such a family of hyperbolic polynomials form a semi-algebraic set which we call a \emph{hyperbolic slice}. We initiate here the study of the geometry of these objects in more detail. The set of hyperbolic polynomials is naturally stratified with respect to the multiplicities of the real zeros and this stratification induces also a stratification on the hyperbolic slices. Our main focus here is on the \emph{local extreme points} of hyperbolic slices, i.e., the local extreme points of linear functionals, and we show that these correspond precisely to those hyperbolic polynomials in the hyperbolic slice which have at most $k$ distinct roots and we can show that generically the convex hull of such a family is a polyhedron. Building on these results, we give consequences of our results to the study of symmetric real varieties and symmetric semi-algebraic sets. Here, we show that sets defined by symmetric polynomials which can be expressed sparsely in terms of elementary symmetric polynomials can be sampled on points with few distinct coordinates. This in turn allows for algorithmic simplifications, for example,to verify that such polynomials are non-negative or that a semi-algebraic set defined by such polynomials is empty.
\end{abstract}
\maketitle

\section{Introduction}
A monic real univariate polynomial $f$ which has only real roots is classically called a hyperbolic polynomial. Such polynomials and their multivariate relatives appear naturally in various mathematical contexts from differential equations to combinatorics, real algebraic geometry and optimization (see for example \cite{gurvits2006hyperbolic,guler1997hyperbolic,marcus,branden}). By identifying monic polynomials of degree $n$ with the list of coefficients, one can describe hyperbolic polynomials of degree $n$ as a semi-algebraic subset of $\R^n$. We consider linear slices, i.e., intersections with linear subspaces, of this semi-algebraic set, which is in fact the closure of one connected component of the complement of the discriminant variety. The study of these \emph{hyperbolic slices} is inspired by the works of Arnold who considered families of hyperbolic polynomials where the first $k$ coefficients were fixed. Arnold \cite{arnol1986hyperbolic} and Givental \cite{givental1987moments} showed that these sets are topologically contractible (see also \cite{meguerditchian1992theorem,meguerditchian1991geometrie}) and have a rich geometric structure as was shown by Kostov \cite{kostov1989geometric} (see also \cite{kostov2002arrangements,kostov2007hyperbolic} for more related results). In a similar spirit to the works of Arnold and Meguerditchian we study the local extreme points of these sets (see Definition \ref{def:extreme}). In analogy to their result, we show in Theorem \ref{maintheorem} that these points correspond to hyperbolic polynomials with few distinct roots. Furthermore, we show in Theorem \ref{thm:2} that a generic hyperbolic slice only has finitely many local extreme points. This signifies in particular that the convex hull of each of its connected components is in fact a polyhedron. In contrast to the case considered by Arnold, our slices are in general not contractible and not compact. However, we are able to give some sufficient condition to decide if a hyperbolic slice is compact or has at least a local extreme point. 

\noindent One of our main interests for the study of these hyperbolic slices stems from an application to symmetric real polynomial functions, i.e., polynomial functions that are left invariant by any permutation of the variables. Real symmetric functions are related to hyperbolic polynomials via the so called \emph{Vieta map}: Recall that for $1\leq i \leq n$ the $i$-th elementary symmetric polynomial in $n$ variables is defined by
\[
e_i:=\sum_{1\leq j_1<j_2<\dots<j_i\leq n}X_{j_1}\cdots X_{j_i}.
\]
By Vieta's formula the coefficients of a univariate monic polynomial of degree $n$ are given by evaluating these elementary symmetric polynomials at the corresponding roots. Conversely, it is also classical that the roots depend continuously on the coefficients and the natural action of $S_n$ permuting the roots does not effect the coefficients. Therefore, the polynomial map from $\R^n$ to $\R^n$ defined by the above connection effectuates a homeomorphism from $\R^n/S_n$ to its image called the \emph{Vieta map}. Since it is classically known that every symmetric polynomial can be uniquely written as a polynomial in the elementary symmetric polynomials one can view real symmetric polynomial functions as functions on the image of the Vieta map. This connection between univariate monic polynomials and symmetric polynomials in $n$ variables gives rise to an application of our results on hyperbolic slices in the context of symmetric polynomial functions: We are interested in the question to what extend the global behavior of symmetric functions is determined by its behavior of symmetrical points or points with a large stabilizer. For example, several authors (e.g. \cite{keilson1967global,waterhouse1983symmetric}) have studied families of symmetric polynomials which attain their minimal values on symmetric points, i.e., points where all coordinates are equal. More generally, it has been shown that symmetric polynomial functions of a given degree $2d$ assume only non-negative values if and only if they have this property on point with at most $d$ distinct coordinates \cite{timofte2003positivity,riener2012degree}. To further this line of ideas, we introduce the notion of $k$-complete symmetric polynomial functions. Those are polynomial functions whose set of values is already obtained by evaluation only on points which have at most $k$ distinct coordinates (see Definition \ref{def:kcomplete}). Using the geometry of hyperbolic slices we are able to identify a new class of $k$-complete functions in Theorem \ref{thm:4} which is given by functions that are constant or linear along a hyperbolic slice (see Definition \ref{def:sprase} for the technical definition). The results we give  here also include the mentioned findings of \cite{timofte2003positivity,riener2012degree} which can be interpreted by saying that every symmetric polynomial of degree $d\geq 4$ is $\floor*{\frac{d}{2}}$-complete. 

\noindent The class of $k$-complete symmetric functions allows for significant algorithmic simplifications in several algorithmic tasks related to polynomial functions. For example, it is known (see \cite{murty1985some}) that checking if a real multivariate polynomial $f$ is non-negative is in general $NP$-hard, already in the case of polynomials of degree 4. However, as we discuss in this article, the complexity of verifying non-negativity for a $k$-complete symmetric polynomial can be drastically reduced if $k<n$, since the set of points that need to be considered is of dimension $k$. We highlight this and several related results in the second part of the article.

\paragraph{\bf Outline:} In Section \ref{sec:hs} we introduce the notion of hyperbolic slices as families of hyperbolic polynomials defined by linear conditions on the coefficients. Our main result in this section is that the local extreme points of such slices correspond to hyperbolic polynomials with few distinct roots (Theorem \ref{maintheorem}) and that generically there are only finitely many such local extreme points (Theorem \ref{thm:2}). Finally, we give sufficient criteria for the existence of such local extreme points in the cases when a slice is not compact.  In Section \ref{sec:positivity} we study symmetric polynomials which attain their minima on points with few distinct coordinates, i.e., on points with a non trivial and potentially large stabilizer. Our main results there (Theorem \ref{thm:4} and Corollary \ref{cor:thm:4}) provide a large class of such functions based on the results from Section \ref{sec:hs}. We furthermore highlight how to efficiently verify that a given symmetric polynomial satisfies the conditions needed to apply these results. The following Section \ref{sec:ex} highlights the applicability of our results. We show that our findings allow for simple proofs for different symmetric inequalities and also recover the mentioned known results. Furthermore, we in particular highlight in Theorem \ref{thm:1suff} a family of symmetric polynomials which attain their minimum on symmetric points. Finally,  we close with some concluding remarks and outlooks in Section \ref{sec:con}.

\paragraph{\bf Notation:} Throughout the article, we fix
$n\in \N$ and denote by 
$\R[\X] := \R[X_1,\dots,X_n]$
the polynomial ring in $n$ variables over $\R$.

\section{Hyperbolic slices}\label{sec:hs}
In this section we define and analyze the notion of a hyperbolic slice. To begin we formalize the notion of hyperbolic polynomials as used in the article.

\begin{definition}
We will denote by 
\[\H:= \left\{ z\in \R^n ~ \middle| ~ T^n-z_1 T^{n-1}+\dots+ (-1)^n z_n \text{ only has real roots} \right\}\]
the set of hyperbolic polynomials of degree at most $n$, and for $1\leq m\leq n$ the \emph{$m$-boundary} of $\H$
\[\H^m:= \left\{ z\in \H ~ \middle| ~ T^n-z_1 T^{n-1}+\dots+ (-1)^n z_n \text{ has at most $m$ distinct roots} \right\}.\]
\end{definition}

As described above we are interested in families of univariate monic hyperbolic polynomials whose coefficients are restricted by linear conditions. In order to define this more concretely, we fix throughout this section an integer $1\leq k\leq n$, a real point $a\in \R^k$, and a surjective linear map $L: ~ \R^n \longrightarrow \R^k$. This choice of a linear map and a point characterizes the linear conditions we aim to impose on hyperbolic polynomials and the \emph{hyperbolic slices} corresponding to these choices can be defined as follows.

\begin{definition}
With the notation introduced above, the \emph{hyperbolic slice} associated to $L$ and $a$ is the affine linear slice
\[\H_L(a):=\H \cap L^{-1}(a).\;\]
Furthermore, for $1\leq m\leq n$ we define by \[\H_L^m(a):=\H^m \cap L^{-1}(a),\] its restriction to the $m$-boundary.
\end{definition}
We briefly discuss one possible connection of the above definition to polynomial interpolation for which our results might be interesting in their own rights: For $k\in \N$ consider $a_1,b_1,\ldots, a_k,b_k\in \R$. Then the space of polynomials $f$ of degree $n$ which satisfy $f(a_i)=b_i$ for $1\leq i\leq k$ is called a polynomial interpolation space. Now, since evaluations at given points define linear maps, an interpolation problem for which one is interested in hyperbolic polynomials only constitutes one example of a hyperbolic slice defined above. 

 Clearly, the assumption that $L$ is surjective is only for convenience in the notation. As mentioned above the set of hyperbolic polynomials is tightly connected to the Vieta map.

\begin{remark}\label{rem:roots}
The set $\H$ of hyperbolic polynomials is the image of the so-called Vieta map
\[\abb{\Gamma}{\R^n}{\H}{x=(x_1,\ldots,x_n)}{(e_1(x),\ldots,e_n(x))},\]
and the restriction of $\Gamma$ to the polyhedral cone \[\W:=\{x\in\R^n\, |\, x_1\leq x_2\leq\ldots\leq x_n\}\] is a homeomorphism. In particular, the roots of a univariate polynomial depend continuously on its coefficients. 
$\H$ is in fact a basic closed semi-algebraic subset of $\R^n$. Clearly, $\H=\H^{n}\supset\H^{n-1}\supseteq\dots\supseteq \H^1$ and $\H^{n-1}$ is the topological boundary of $\H$. Furthermore, for $1\leq m\leq n$ the $m$-boundary $\H^m$ is the image of the union of the $m$-faces of $\W$ under $\Gamma$ and therefore of dimension $m$. For more details, we refer to \cite[Appendix V.4]{Whitney}.
\end{remark}
The next example shows one of the simplest situations of a hyperbolic slice obtained by fixing the first two coefficients of a monic polynomial of degree 4. 
\begin{example}
For $k\geq 2$ we can fix the first $k$ coefficients of a monic polynomial. The set of hyperbolic polynomials in such a family defines a hyperbolic slice and this setup corresponds to the situation studied by Arnold \cite{arnol1986hyperbolic} and Kostov \cite{kostov1989geometric}.
For example, we can consider $\H_L(0,-6)$, where
 \[\abb{L}{\R^4}{\R^2}{(z_1,z_2,z_3,z_4)}{(z_1,z_2)}.\] This choice yields the hyperbolic slice in the plane shown in Figure \ref{fig:slice1}.
\begin{figure}[h]
 \centering
 \includegraphics[width=0.3\textwidth]{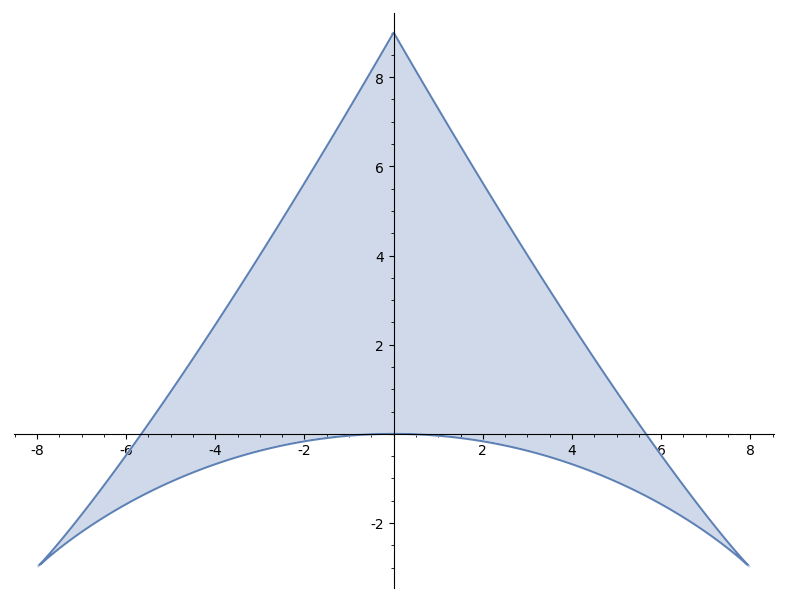}\vspace{-5mm}
 \caption{The hyperbolic slice $\H_L(0,-6)$}
 \label{fig:slice1}
\end{figure}
\end{example}
As can be seen from the example above, a hyperbolic slice is not convex but bears some resemblance to a polytope. By the connection via the Vieta map, we have that $\mathcal{H}$ is homeomorphic to the polyhedral cone $\mathcal{W}$. Furthermore, one finds three extreme points/ vertices in the above picture. For convex sets in $\R^n$ the extreme points contain important information about the set. To generalize this notion to the sets defined above, we will be interested in the following local notion of extreme points. 
 \begin{definition}\label{def:extreme}
Let $A\subseteq \R^n$. We call $z\in A$ a \emph{local extreme point} of $A$, if there is a neighborhood $U\subseteq \R^n$ of $z$ such that $z$ is an extreme point of $\conv(A\cap U)$. We denote the set of all local extreme points of $A$ by $\locextr (A)$.
\end{definition}
Classically, in convex optimization, the interest in extreme points stems from the fact that linear functions attain their minimum or maximum on these points. Similarly, the following holds for local extreme points.
\begin{remark}\label{rem:2}
Let $A\subseteq \R^n$, and $\varphi\in \Hom(\R^n,\R)$ and $z_\varphi\in A$ a (strict) local minimal point of $\varphi$ in $A$. Then $z_\varphi$ is also a local extreme point of $A$. Conversely, let $z\in A$ be a local extreme point of $A$, then there is $\varphi_z\in \Hom(\R^n,\R)$ such that $z$ is a local minimal point of $\varphi$ in $A$.
\end{remark}
\begin{example}
We more generally examine the local extreme points of the hyperbolic slices discussed above which are similar to the one in Figure \ref{fig:slice1}. We consider again the linear map 
\[\abb{L}{\R^4}{\R^2}{(z_1,z_2,z_3,z_4)}{(z_1,z_2)},\] and we examine local extreme points of the family of slices $\H_L(0,a)$, with $a\in \R$. Then we find that the local extreme points in this case are
\[\locextr(\H_L(0,a))=\H_L^2(0,a) =\left\{ \left(0,a, \pm \left(\sqrt{-\frac{2a}{3}}\right)^3,-\frac{a^2}{12}\right) , \left(0,a,0,\frac{a}{2}\right) \right\}.\]
By examining the resultants of the corresponding quartic polynomials and their second derivative, one finds that each of these local extreme points corresponds to hyperbolic polynomials with at most two distinct roots.

\end{example}

As a first result, we are now going to establish that the above example generalizes in the following sense. For a general hyperbolic slice, defined through $k$ linear conditions, the local extreme points can be characterized as hyperbolic polynomials of the $k$-boundary. This generalizes 
Theorem \cite[Theorem 4.2]{riener2012degree} to general hyperbolic slices. 

\begin{theorem}\label{maintheorem}
The local extreme points of a hyperbolic slice are contained in the $k$-boundary, i.e., \[\locextr (\H_L(a)) \subseteq \H_L^k(a).\]
\end{theorem}
\begin{proof}
Let $z\in \H_L(a)$ be a local extreme point, i.e., there is a neighborhood $U$ of $z$ such that $z$ is an extreme point of $\conv(\H_L(a)\cap U)$.
We assume that $z\notin \H_L^k(a)$ and want to find a contradiction. To this end, we want to find $c\in \ker L$ non-zero such that $z\pm \varepsilon c \in \H_L(a)$ for all $\varepsilon >0$ small enough. 
Consider $f:=T^n-z_1T^{n-1}+\dots +(-1)^n z_n$ with distinct roots $x_1,\dots,x_m$ where $m>k$ and factor as follows:
\[f = \underbrace{\prod_{i=1}^m (T - x_i)}_{=:p}\cdot q,\]
where the set of zeros of $q$ contains only elements from $\{x_1,\dots,x_m\}$ and $q$ is of degree $n-m$. Write $q=T^{n-m}+q_1 T^{n-m-1}+\dots +q_{n-m}$ and define $q_0:=1$ and consider the linear map
\[\abb{\chi}{\R^m}{\R^n}{y}{\left(\sum_{i+j=1}q_i y_j,\dots,\sum_{i+j=n}q_i y_j\right)}.\]
Since $m>k$, there is $b\in \ker (L\circ \chi)\setminus \{0\}$. We define $h:=b_1T^{m-1}+\dots+b_{m}$ and $g:=h\cdot q=c_1T^{n-1}+\ldots +c_{n}\neq 0$, where $c=\chi (b)$ by construction and therefore $c\in \ker L$. Now, because $p$ has no multiple roots, $p \pm \varepsilon h$ is hyperbolic for $\varepsilon>0$ small enough: the roots depend continuously on the coefficients and complex roots come as conjugated pairs (see Remark \ref{rem:roots}). Hence
\[(p \pm \varepsilon h)\cdot q = f \pm \varepsilon h\cdot q = f \pm \varepsilon g\]
is hyperbolic for all $\varepsilon>0$ small enough, i.e., $z\pm \varepsilon c \in \H_L(a)$.
If we choose $\varepsilon >0$ small enough we can ensure also that $z\pm\varepsilon c \in U$. But then 
\[z = \frac{z + \varepsilon c + z - \varepsilon c}{2},\]
a contradiction to $z$ being an extreme point of $\conv(\H_L(a)\cap U)$.
\end{proof}
\begin{remark}
If the map $L$ is not surjective, one can obtain similar results by replacing $k$ with $\rank L$.
\end{remark}

In view of Remark \ref{rem:2} we get the following.
\begin{corollary}\label{corollarycompact}
Let 
$g:\ \R^n\to \R$ be a linear or concave function  and consider the optimization problem
\[\min_{z\in \H_L(a)} g(z).\]
Let $M$ denote the set of minimizers of this problem. If $\H_L(a)$ is non-empty and compact, then we have $M \cap H_L^k(a) \neq \emptyset$. In particular $H_L(a)$ contains a point $z\in \H_L^k(a)$.
\end{corollary}
\begin{proof}
Since $\H_L(a)$ is compact, there is a minimizer $z\in M$ such that $z$ is an extreme point of the convex hull of $\H_L(a)$. In particular, $z$ is a local extreme point of $\H_L(a)$ and therefore on the $k$-boundary of $\H_L(a)$ by Theorem \ref{maintheorem}, i.e., $z\in M\cap \H_L^k(a)$.
\end{proof}

As can be observed in the example shown in Figure \ref{fig:slice1} connected components of hyperbolic slices appear to have a similarity to polytopes. They are not convex, but appear to be ``deflated'' polytopes. To make this a bit more concrete we show that a generic hyperbolic slice has only finitely many local extreme points. This in particular implies that their convex hull, or in fact the convex hull of each of its connected components, is a polytope. The proof uses elementary properties of subdiscriminants. The relevance of subdiscriminants for counting roots of real univariate polynomials is explained in \cite[Chapter 4]{Basu2003}.

\begin{definition} Let $f \in \R[T]$ be a monic polynomial of degree $n$ with roots $x_1,\dots,x_n$ in $\C$. Then the \emph{$(n-m)$-subdiscriminant}, $1\leq m \leq n$, of $f$ is defined as
\[\sDisc_{n-m}(f) = \sum_{\substack{I\subseteq\{1,\dots,n\}\\ |I|=m}} \prod_{\substack{i,j\in I\\ j>i}} (x_i-x_j)^2. \]
\end{definition}
\begin{remark}\label{rem:subdisc}
Each $(n-m)$-subdiscriminant of $f$ is defined above as a polynomial of degree $m(m-1)$ in terms of the roots of $f$. Noticing that each of the expressions is in turn symmetric in the roots, one immediately obtains that each subdiscriminant of $f$ can be expressed in the elementary symmetric polynomials evaluated at the roots, i.e., in the coefficients of $f$. Indeed, the subdiscriminants of $f$ can be obtained directly by minors of the Sylvester matrix - also called subresultants - of $f$ and $f'$. So the degree of each $(n-m)$-subdiscriminant expressed in the coefficients is $2m-2$ \cite[Proposition 4.27]{Basu2003}.
\end{remark}
\begin{proposition}\cite[Remark 4.6 and Proposition 4.50]{Basu2003}\label{prop:subdisc}
A monic polynomial $f \in \R[T]$ of degree $n$ has exactly $k$ distinct roots if and only if
\[\sDisc_0(f)=\dots=\sDisc_{n-k-1}(f)=0,~\sDisc_{n-k}(f)\neq 0.\]
Moreover, if and only if additionally
\[ \sDisc_{n-k}(f) > 0,\dots, \sDisc_{n-1}(f)>0,\]
then $f$ has only real roots.
\end{proposition}

\begin{theorem}\label{thm:2}
The $k$-boundary $\H^k_L(a)$ of a generic hyperbolic slice is finite. In particular, a generic hyperbolic slice has only finitely many local extreme points. The number of those points is bounded by
\[\min\left\{2^{n-k}\frac{(n-1)!}{(k-1)!}, \binom{n}{k} \frac{(n-1)!}{(k-1)!}\right\}.\]
\end{theorem}
\begin{proof}
First, we establish that for a generic hyperbolic slice the $k$-boundary $\H^k_L(a)$ is finite. 
For this recall that the set of hyperbolic polynomials with at most $k$ distinct roots, $\H^k$, is of dimension $k$ by Remark \ref{rem:roots}. 
Therefore, a generic $(n-k)$-dimensional affine linear subspace will intersect $\H^k$ in only finitely many points. Furthermore, in view of Proposition \ref{prop:subdisc} we see further that $\H^k$ is contained in the algebraic set defined by the vanishing of $n-k$ polynomials. On the one hand, each of the subdiscriminants describing this algebraic set is a homogeneous polynomial of degree $(2n-2), (2n-4),\ldots,(2k)$ expressed in the elementary symmetric polynomials by Remark \ref{rem:subdisc} and we can apply Bézout's Theorem to obtain the bound 
\[2^{n-k}\frac{(n-1)!}{(k-1)!}.\]
On the other hand, we can apply the weighted Bézout's Theorem (see \cite[chapter VIII]{Mondal2021}): We assign to the $i$-th elementary symmetric polynomial $e_i$ the weight $i$. Then each subdiscriminant is weighted homogeneous of degree $n(n-1),(n-1)(n-2),\dots,(k+1)k$. Indeed, this is exactly the degree of the subdiscriminants expressed in the roots. Furthermore, we can bound the weighted degree of each of the $k$ affine hyperplanes describing our slice by $n,n-1,\dots,n-k+1$. So we obtain the bound
\[\frac{1}{n!}\frac{n!}{(n-k)!}\cdot \frac{n!(n-1)!}{k!(k-1)!}=\binom{n}{k} \frac{(n-1)!}{(k-1)!}.\] 
\end{proof}

\begin{remark}
The second bound obtained in \ref{thm:2} by the weighted Bézout's Theorem can even be refined, when one considers the coefficients appearing in $L(z)$ for $z\in \H$. For example, if just the first coefficients are fixed, i.e., $L(z)=(z_1,\dots,z_k)$, then $\binom{n}{k}$ can be replaced by $1$.
\end{remark}

Since the extreme points of the convex hull of a set are local extreme points, we can deduce the following.

\begin{corollary}
The convex hull of a generic hyperbolic slice is a polyhedron. The same applies to any of its connected components.
\end{corollary}

Note that the proof of Theorem \ref{thm:2} together with Proposition \ref{prop:subdisc} gives an explicit description of the $k$-boundary of a hyperbolic slice as a semi-algebraic set. The following example shows that the $k$-boundary of a hyperbolic slice can be infinite. But even in this case, there might only be finitely many local extreme points.

\begin{example}
Consider $L:~\R^4\to \R^3, ~ (z_1,z_2,z_3,z_4)\mapsto (z_1,z_3,z_4)$ and $a\in \R$. Then
\[\H_L(a,0,0)=\left\{ (a,z_2,0,0) ~\middle|~ z_2\in \R, ~z_2\leq \frac{a^2}{4} \right\} = \H_L^3(a,0,0)\]
is not finite. But $\H_L(a,0,0)$ is obviously convex with only local extreme point
\[\left(a,\frac{a^2}{4},0,0\right)\in \H_L^2(a,0,0).\]

\end{example}

Next, we will give sufficient conditions on $L$ for the compactness of a hyperbolic slice and for the existence of local extreme points. For that, we will need the following definition.

\begin{definition}
Let $f,g\in \R[T]$ be hyperbolic polynomials with real roots $\alpha_n \leq\cdots \leq \alpha_1$ and $\beta_{m} \leq\ldots\leq \beta_1$ respectively. We say that $g$ \emph{interlaces} $f$ if $\alpha_n\leq\beta_{m}\leq\alpha_{n-1}\leq\ldots\leq \alpha_1$ or $\beta_m\leq\alpha_n\leq\beta_{m-1}\leq\ldots\leq \alpha_1$. Furthermore, we say $f$ and $g$ are \emph{interlacing}, if $f$ interlaces $g$ or $g$ interlaces $f$.
\end{definition}

\begin{remark}
If $g$ interlaces $f$, then clearly $f$ and $g$ either have the same degree, i.e., $n=m$ or the degree of $g$ is smaller by one, i.e., $m=n-1$. 
\end{remark}
The following classical result (see \cite[Theorem 4.1.]{dedieu1992obreschkoff}) connects interlacing polynomials to linear pencils of hyperbolic polynomials.
\begin{theorem}[Dedieu]\label{dedieu}
Let $f,g\in \R[T]$ be hyperbolic, non-zero polynomials of degree at most $n$. Then the following statements are equivalent:
\begin{enumerate}
 \item $f$ and $g$ are interlacing.
 \item $f+\xi\cdot g$ is hyperbolic for any $\xi\in \R$.
\end{enumerate}
\end{theorem}
From now on we express $L$ in terms of $k$ linearly independent linear forms $l_1,\dots,l_k \in \R[Z_1,\dots,Z_n]_1$ as $L:~ \R^n\to \R^k, ~ z \mapsto (l_1(z),\dots,l_k(z))$. We can use the results above to give a sufficient condition on $l_1,\dots,l_k$ for the existence of local extreme points of a hyperbolic slice.

\begin{lemma}\label{lemmalocalextremepoint}
If $Z_1\in \Span(l_1,\dots,l_k)$ and $\H_L(a)\neq \emptyset$, then $\H_L(a)$ has a local extreme point.
\end{lemma}
\begin{proof}
Let $z \in \H_L(a)$ and write $Z_1=\sum_{i=1}^k \lambda_i l_i$ for some $\lambda_1,\dots,\lambda_k \in \R$. Furthermore, denote by $x=(x_1,\dots,x_n)\in \R^n$ the roots of 
\[f_z := T^n - z_1 T^{n-1} + \cdots +(-1)^n z_n.\] 
Then $e_1(x) = z_1 =\sum_{i=1}^k \lambda_i l_i(z)=\sum_{i=1}^k \lambda_i a_i$ and hence
\[z_2 = e_2(x) = \frac{1}{2}\left(e_1(x)^2 - \sum_{i=1}^n x_i^2\right) \leq \frac{1}{2} e_1(x)^2 = \frac{1}{2}\left(\sum_{i=1}^k \lambda_i a_i\right)^2.\]
So the optimization problem
\[ \max_{z\in \H_L(a)} z_2 \]
has a non-empty set of maximizers $M$. Suppose $\H_L(a)$ has no local extreme point. Then $M$ contains a line, i.e., there is a maximizer $m=(m_1,\dots,m_n)\in M$ and a $y=(y_1,\dots,y_n)\in \R^n$ non-zero such that $y_1=y_2=0$ and $m+\xi y\in \H$ for all $\xi\in \R$. This means $f:=T^n-m_1T^{n-1}+\cdots+(-1)^n m_n$ and $g:=-y_3 T^{n-3}+\cdots+(-1)^n y_n$ are interlacing by \ref{dedieu}, which is not possible because of degree reasons.
\end{proof}
We can use the existence of an extreme point, for example, to obtain the following result which connects to polynomial interpolation.

\begin{corollary}
Consider the set of polynomials of degree $n$, which are monic, have the second coefficient fixed, and solve a $k$-points interpolation problem. Then there exists a hyperbolic polynomial in this set if and only if there exists one with at most $k$ distinct roots. 
\end{corollary}
\begin{proof}
Under the conditions, the corresponding hyperbolic slice has at least one extreme point by Lemma \ref{lemmalocalextremepoint}.
\end{proof}
By prescribing not only the first but also the second-highest coefficient of a monic polynomial, one directly obtains a sufficient condition for the compactness of a hyperbolic slice.

\begin{lemma}\label{compact}
If $Z_1,Z_2\in \Span(l_1,\dots,l_k)$, then $\H_L(a)$ is compact.
\end{lemma}
\begin{proof}
As the empty set is compact we can assume that there is $z \in \H_L(a)$. Furthermore we write $Z_1=\sum_{i=1}^k \lambda_i l_i$ and $Z_2=\sum_{i=1}^k \chi_i l_i$ for some $\lambda_1,\dots,\lambda_k,\chi_1,\dots,\chi_k \in \R$ and denote by $x=(x_1,\dots,x_n)\in \R^n$ the roots of 
\[f_z := T^n - z_1 T^{n-1} + \cdots +(-1)^n z_n.\] 
Then $e_1(x) = z_1 =\sum_{i=1}^k \lambda_i l_i(x)=\sum_{i=1}^k \lambda_i a_i$ and $e_2(x) = \sum_{i=1}^k \chi_i a_i$ and hence
\[\sum_{i=1}^n x_i^2 = e_1(x)^2 - 2e_2(x) = \left(\sum_{i=1}^k \lambda_i a_i\right)^2 - \sum_{i=1}^k \chi_i a_i.\]
This shows that $x$ is contained in a ball, thus $\H_L(a)$ is bounded. Furthermore, as the roots of a polynomial depend continuously on the coefficients it is clear that $H_L(a)$ is closed and therefore compact (see Remark \ref{rem:2}).
\end{proof}
We close this section with a selection of examples of two-dimensional hyperbolic slices which highlight the various mentioned scenarios.

\begin{example}\label{example3}
Consider $\H_L(a_2,a_4)$, where $a:=(a_2,a_4)\in \R^2$ such that $a_2< 0$ and $a_4> 0$ and
\[\abb{L}{\R^4}{\R^2}{(z_1,z_2,z_3,z_4)}{(z_2,z_4)}.\]
Then, there are the following three possible situations.
\paragraph{\bf a:}
 If $a:=(a_2,a_4)$ satisfy $a_2^2-4a_4< 0$, the hyperbolic slice $\H_L(a)$ will contain two local extreme points. In particular, $\H_L^2(a)\neq \emptyset$. Furthermore, the local extreme points of $\H_L(a)$ are not global extreme points. Therefore, they are not extreme points of the convex hull of $\H_L(a)$. This is illustrated in Figure \ref{fig:a}.
\paragraph{\bf b:} For all values $a:=(a_2,a_4)$ with $a_2^2-4a_4= 0$, $\H_L(a)$ will contain no local extreme points. But the $2$-boundary of $\H_L(a)$ is non-empty. Indeed,
\[T^4+a_2T^2+a_4=\left(T-\sqrt{\frac{-a_2}{2}}\right)^2\left(T+\sqrt{\frac{-a_2}{2}}\right)^2,\]
and thus $(0,a_2,0,a_4)\in \H_L^2(a)$. This situation is illustrated in Figure \ref{fig:b}.

\paragraph{\bf c:} For the values $a:=(a_2,a_4)$ with $a_2^2-4a_4> 0$, $\H_L(a)$ will contain no local extreme point. Moreover, $\H_L^2(a)$ is empty in this case, while $\H_L(a)\neq \emptyset$. This is illustrated in Figure \ref{fig:c}.

\noindent Indeed, the polynomial $f=T^4+a_2T^2+a_4$ is hyperbolic with the $4$ distinct roots \[x_{1,2,3,4}:=\pm \sqrt{\frac{-a_2 \pm \sqrt{a_2^2-4a_4}}{2}}.\] Therefore, the hyperbolic slice $\H_L(a)$ is non-empty. 
On the other hand, suppose that the $2$- boundary $\H_L^2(a)$ is non-empty, i.e., that we can find $(a_1,a_2,a_3,a_4)\in \H_L^2(a)$. This in turn implies that there are $x,y\in \R$ such that the polynomial \[f_a:=T^4-a_1T^3+a_2T^2-a_3T+a_4\] factors either as \[f_a=(T-x)^3(T-y) \text{ or } f_a=(T-x)^2(T-y)^2.\] 

\noindent In the first case a comparison of coefficients shows $a_2=3xy+3x^2$ and $a_4x^3y$. Since $a_4>0$ we must have $x,y\neq 0$ and can solve $y=\frac{a_4}{x^3}$. This implies $a_2=\frac{3a_4}{x^2}+3x^2$ and $3x^4-a_2x^2+3a_4=0$. However, since $x\neq 0$, $a_2< 0$ and $a_4> 0$ we must have $3x^4-a_2x^2+3a_4>0$, and thus have a contradiction. Analogously, for the second case, comparing coefficients shows $a_2=4xy+x^2+y^2$ and $a_4=x^2y^2$. We solve for $y$ and get $y=\pm \frac{\sqrt{a_4}}{x}$ from which we find $a_2=\frac{a_4}{x^2}+x^2\pm 4\sqrt{a_4}$. But since $a_2< 0$, $a_4> 0$ and $a_2^2-4a_4> 0$ the resulting polynomial equation $x^4+(\pm 4\sqrt{a_4}-a_2)x^2+a_4=0$ clearly has no real solution. 

\begin{figure}[ht!]
 \centering
 \begin{subfigure}[b]{0.32\textwidth}
   \centering
   \includegraphics[width=\textwidth]{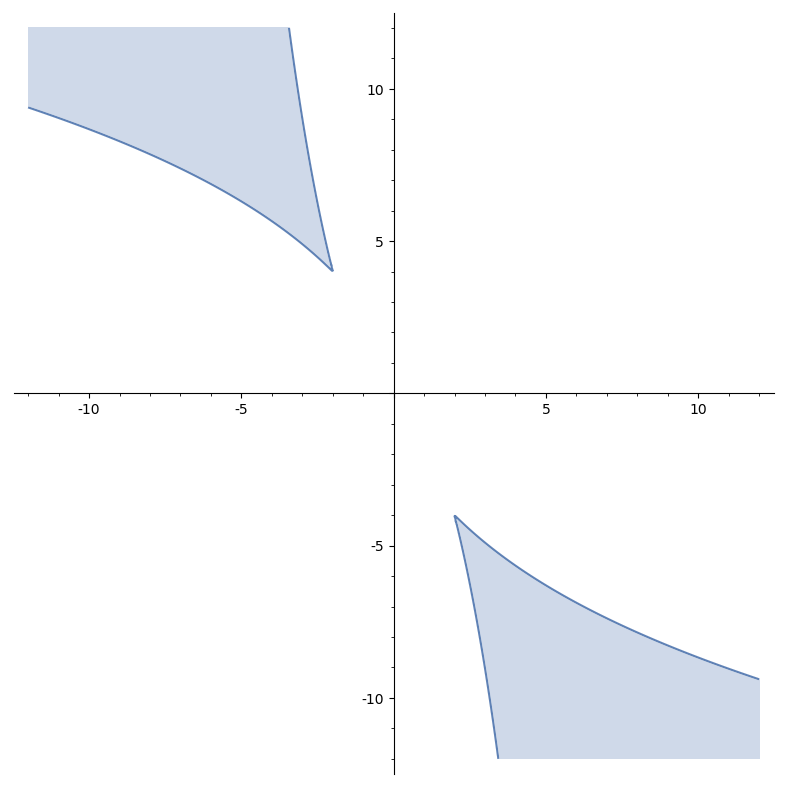}
   \caption{$\H_L(-3,4)$ contains the local extreme points $(\pm 2,-3,\mp 4,4)$.}\label{fig:a}
 \end{subfigure}
  \hfill
 \begin{subfigure}[b]{0.32\textwidth}
   \centering
   \includegraphics[width=\textwidth]{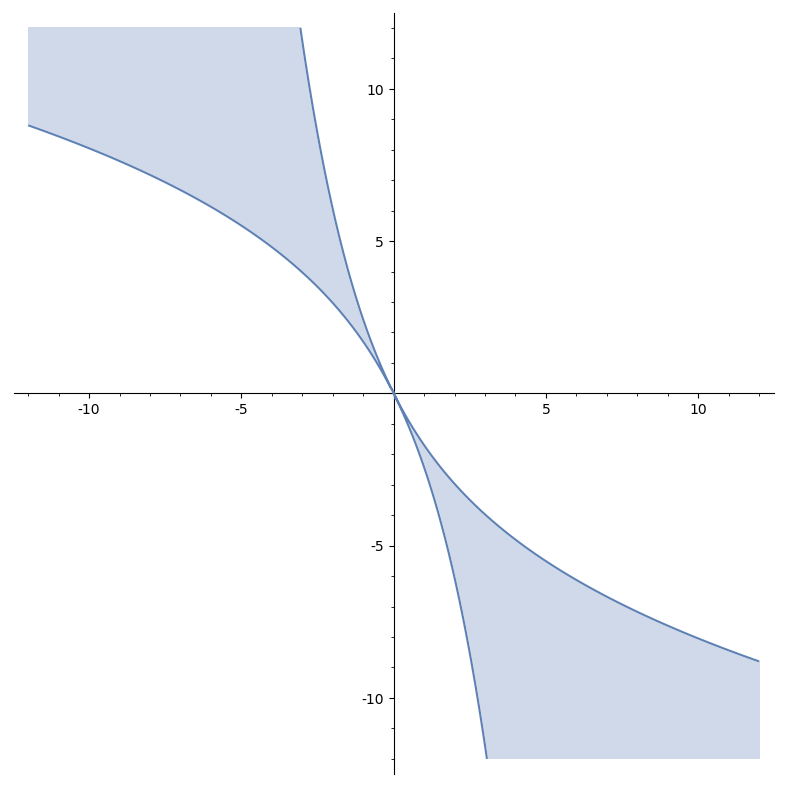}
   \caption{$\H_L(-4,4)$ has no local extreme point and $(0,-4,0,4)\in \H^2$.}\label{fig:b}
 \end{subfigure}
 \hfill
 \begin{subfigure}[b]{0.32\textwidth}
   \centering
   \includegraphics[width=\textwidth]{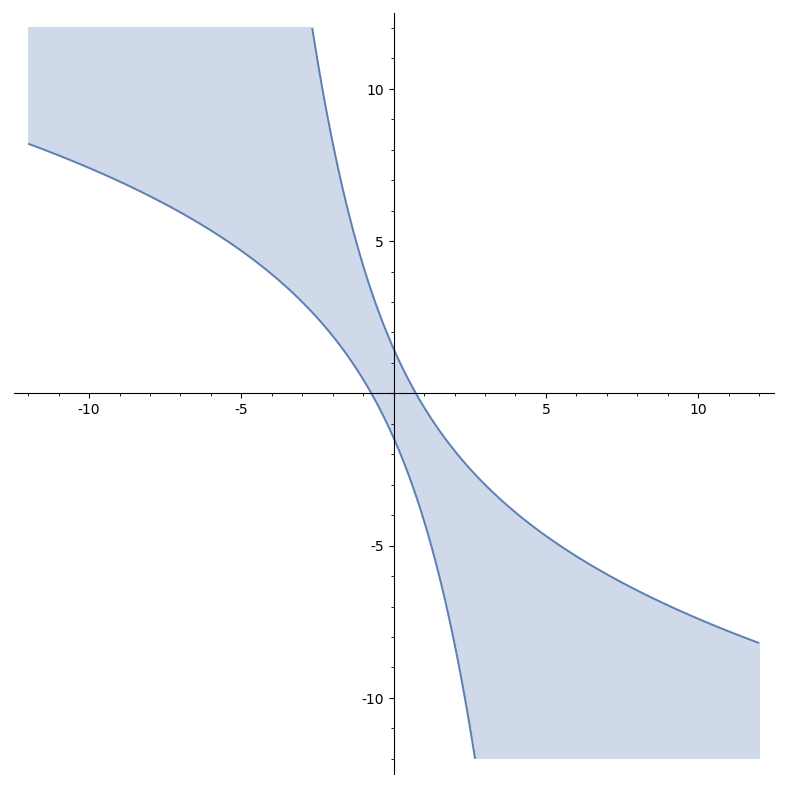}
   \caption{$\H_L(-5,4)$ has no local extreme point and $\H_L^2(-5,4)=\emptyset$.}\label{fig:c}
 \end{subfigure}
\end{figure}
\end{example}

\section{Positivity of symmetric polynomial functions}\label{sec:positivity}
In this section we will study real polynomial functions defined by symmetric polynomials. Since every symmetric polynomial can be written in a unique way as a polynomial in elementary symmetric polynomials, we can use the geometric description of hyperbolic slices obtained before to characterize the minimal points of a large class of symmetric polynomial functions which are sparse in an appropriate sense (see Definition \ref{def:sprase}). It had already been observed by various authors that certain symmetric functions attain their minimal values on symmetric points (e.g. \cite{keilson1967global,foregger1987relative,kovavcec2012note}). Other authors found that symmetric polynomial functions of a bounded small enough degree attain their minima on points with few distinct coordinates (e.g. \cite{timofte2003positivity,riener2012degree}). We generalize these results by considering symmetric polynomial functions which are completely characterized through their values on points with at most $k$ distinct coordinates.

\subsection{The notions of $k$-completeness and $k$-testability}
\begin{definition}\label{def:kcomplete}
For $k\in \N$ we consider the set \[\A_k:=\left\{x\in \R^n \,:\, |\{x_1,\dots,x_n\}| \leq k \right\}\] of points with at most $k$ different coordinates. Given a symmetric polynomials $f\in \R[\X]$ and $S\subseteq \R^n$ we say that $f$ is 
\begin{enumerate}
 \item $k$-\emph{complete} on $S$ if 
 \[f(S)=f(S\cap \A_k).\]
 \item $k$-\emph{testable} on $S$ if 
 \[\inf_{x\in S} f(x)=\inf_{x\in S\cap \A_k} f(x).\]
\end{enumerate}
In case $S=\R^n$ we may omit it and just speak of $k$-testable and $k$-complete polynomials. 
\end{definition}
The two notions of $k$-complete and $k$-testable are very closely connected, but the first one is stronger, while the second one might be interesting in particular in the context of optimization. 
In order to motivate the study of this class, we exemplify first how algorithmic problems can be substantially simplified for $k$-complete and $k$-testable symmetric polynomials.

\begin{definition}
A decreasing sequence of positive integers $\lambda=(\lambda_1,\ldots,\lambda_k)$ which sums up to $n$ is called a \emph{partition of $n$ into $k$ parts}. We will write $\lambda\vdash_k n$ to denote that $\lambda$ is a partition of $n$ into $k$ parts. Let $f\in\R[\X]$ be a symmetric polynomial. Then for $\lambda\vdash_k n$ we define
\[f^\lambda:=f(\underbrace{X_1,\dots,X_1}_{\lambda_1-\text{times}},\dots,\underbrace{X_k,\dots,X_k}_{\lambda_k-\text{times}})\in\R[X_1,\dots,X_k].\]
\end{definition}
Note that the number of partitions of $n$ into $k$ parts is at most $\binom{n+k}{k}$ and thus polynomial in $n$ for a fixed $k$. Therefore the above notion allows reducing, for example, the question of whether a symmetric polynomial in $n$ variables is non-negative to a polynomial number of such queries in $k$ variables. It is, for example, known to be NP-hard to decide the non-negativity of a given polynomial of degree 4 (see e.g. \cite{blum1998complexity} or \cite{murty1985some}). Clearly, by applying the above procedure, one can obtain algorithmic simplifications which yield polynomial complexity for this kind of problem (see also \cite{faugere2023computing} where this method is applied also for other algorithmic questions). We highlight in particular the following version of Artin's solution to Hilbert's 17th problem for $k$-complete symmetric polynomials, which is a direct consequence of the sketched procedure of identifying variables. 

\begin{proposition}[Hilbert's 17th problem for $k$-complete polynomials]
Let $f\in\R[\X]$ be a symmetric $k$-testable polynomial. Then $f$ attains only non-negative values on $\R^n$ if and only if for all $\lambda\vdash_k n$ we can find a sum of squares of polynomials $t\in\sum\R[X_1,\dots,X_k]^2$ such that $t\cdot f^\lambda$ is also a sum of squares of polynomials. 
\end{proposition}

The main interest in the statements presented above is that the reduction of dimension also gives new complexity bounds for the degrees of the polynomials in question. For example, for Hilbert's 17th problem for $k$-complete polynomials we can adapt the currently known complexity bounds.
\begin{remark}
Let $f$ be a $n$-variate $k$-complete polynomial of degree $d$. Then $f$ is non-negative if and only if we can write each $f^\lambda$ as a sum of at most $2^k$ rational squares by \cite{pfister1967darstellung}. We can also write each $f^\lambda$ as a sum of squares of rational functions, where, following \cite{lombardi2014elementary}, we obtain the following degree bounds for the numerators and denominators:
\[2^{2^{2^{d^{4^{k}}}}}.\]
\end{remark}

\subsection{Sufficient and quasi-sufficient polynomials}

Now, we want to show that it is possible to produce a large class of $k$-complete symmetric polynomials based on the results on hyperbolic polynomials. Throughout this section we fix $1\leq k\leq n$ and consider the $k$ linearly independent linear forms $l_1,\dots,l_k \in \R[Z_1,\dots,Z_n]_1$ and the linear map $L:~ \R^n\rightarrow \R^k, ~ z \mapsto (l_1(z),\dots,l_k(z))$. Recall that a symmetric polynomial $f\in \R[\X]$ can be written uniquely in terms of the elementary symmetric polynomials, say $f=g(e_1,\dots,e_n)$. Now evaluation of $f$ in a point $x\in \R^n$ translates into evaluation of $g$ in a point $z\in \H$ and evaluation on $\A_k$ translates into evaluation of $g$ on $\H^k$. By partitioning 
\[\H=\bigcup_{a\in \R^k} \H_L(a) ~ \text{and} ~ \H^k=\bigcup_{a\in \R^k} \H^k_L(a)\]
for the map $L$, we can use our previous results to show under some mild conditions that $f$ is $k$-complete or $K$-testable if it allows for a special representation in terms of $k$ linear forms of elementary symmetric polynomials. We define these representations in the following. 
\begin{definition}\label{def:sprase}
Let $f \in \R[\X]$ be a symmetric polynomial and write $f$ in terms of elementary symmetric polynomials, say $f=g(e_1,\dots,e_n)$ for some $g\in \R[Z_1,\dots,Z_n]$.
\begin{enumerate}
\item We say that $f$ is \emph{$(l_1,\dots,l_k)$-sufficient} if $g\in\R[l_1,\dots,l_k].$
\item We say that $f$ is \emph{$(l_1,\dots,l_k)$-quasi-sufficient} if $f$ admits a representation of the form \[f=f_0+f_1e_1+\dots+f_ne_n\] for some $(l_1,\dots,l_k)$-sufficient polynomials $f_0,\dots,f_n$.
\item Furthermore, we say that $f$ is \emph{$(l_1,\dots,l_k)$-concave-sufficient} if $g$ is concave on $H_L(a)$ for all $a\in \R^k$.
\end{enumerate}
Moreover, we say that a symmetric semi-algebraic set $S\subseteq \R^n$ is \emph{$(l_1,\dots,l_k)$-sufficient}, if it can be described by $(l_1,\dots,l_k)$-sufficient polynomials.
\end{definition}
The following proposition is a direct consequence of the unique representation of a symmetric polynomial of degree $d$ in terms of the elementary symmetric polynomials and may serve as a motivation for the definitions given above. 
\begin{proposition}\label{prop:(half-)degree}
Let $f \in \R[\X]$ be symmetric of degree $d$. Then $f$ is $\left( Z_1,\dots,Z_{d} \right)$-sufficient and $\left( Z_1,\dots,Z_{\floor*{\frac{d}{2}}} \right)$-quasi-sufficient. 
\end{proposition}
\begin{remark}\label{rem:implications}
The notions defined above are increasingly strict in the following sense: Sufficiency $(1)$ implies quasi-sufficiency $(2)$, which in turn implies concave-sufficiency $(3)$ of both $f$ and $-f$.
\end{remark}
The results on hyperbolic slices now translate to the following statements on symmetric real polynomial functions.

\begin{theorem}\label{thm:4}
Let $S\subseteq \R^n$ be a symmetric $(l_1,\dots,l_k)$-sufficient semi-algebraic set and let $f\in\R[\X]$ be a symmetric polynomial.
\begin{enumerate}
 \item If $f$ is $(l_1,\dots,l_k)$-sufficient and if every non-empty hyperbolic slice $\H_L(a)$ contains a local extreme point, then $f$ is $k$-complete on $S$.
 \item If $f$ is $(l_1,\dots,l_k)$-concave-sufficient and $\H_L(a)$ is compact for all $a \in \R^k$, then $f$ is $k$-testable on $S$.
 \item If $f$ is $(l_1,\dots,l_k)$-quasi-sufficient and $\H_L(a)$ is compact for all $a \in \R^k$ and $S\cap \A_k$ is connected, then $f$ is $k$-complete on $S$.
 \item If $f$ is $(l_1,\dots,l_k)$-concave-sufficient and not $(l_1,\dots,l_k)$-sufficient and \[\inf_{x\in S} f(x)>-\infty,\]
 then $f$ is $k$-testable on $S$.
\end{enumerate}
\end{theorem}
\begin{proof}
$(1)$: Let $g\in\R[Z_1,\dots,Z_n]$ such that $f = g(e_1,\dots, e_n)$. Let $x\in S$ and consider $z:=\Gamma(x)$ and $a:=L(z)$. There is $\tilde z \in \H_L^k(a)$ by Theorem \ref{maintheorem} since $\H_L(a)$ admits a local extreme point. So there is $\tilde x \in \A_k$ with $\Gamma (\tilde x)=\tilde z$. Then $f(x)=f(\tilde x)$ and $\tilde x \in S$ since $f$ and $S$ are $(l_1,\dots,l_k)$-sufficient.

$(2)$: Let $g\in\R[Z_1,\dots,Z_n]$ such that $f = g(e_1,\dots, e_n)$. Let $x\in S$ and consider $z:=\Gamma(x)$ and $a:=L(z)$. Since $g$ is concave on $L^{-1}(a)$ by the concave-sufficiency of $f$ and $\H_L(a)$ is compact we can apply Corollary \ref{corollarycompact} and get that
\[\min_{y\in \H_L(a)} g(y)=\min_{y\in \H^k_L(a)} g(y),\]
i.e., there is $\tilde z \in \H^k_L(a)$ with $g(\tilde z)\leq g(z)$. Let $\tilde x \in \A_k$ with $\Gamma (\tilde x)=\tilde z$. Then $f(\tilde x)\leq f(x)$ and $\tilde x \in S$ since $S$ is $(l_1,\dots,l_k)$-sufficient and we can conclude that $f$ is $k$-testable on $S$.

$(3)$: Let $x_0\in S$. We can apply $(2)$ since $f$ and $-f$ are both $(l_1,\dots,l_k)$-concave-sufficient by Remark \ref{rem:implications} and get that
\[\inf_{x\in S} f(x) = \inf_{x\in S\cap\A_k} f(x) \quad \text{and} \quad \sup_{x\in S} f(x) = \sup_{x\in S\cap\A_k} f(x),\]
so there are $x_1,x_2\in S\cap\A_k$ with $f(x_1)\leq f(x_0)$ and $f(x_2)\geq f(x_0)$. Since $S\cap \A_k$ is connected there is $\tilde x \in S\cap \A_k$ with $f(\tilde x)=f(x_0)$ by the intermediate value theorem.

$(4)$: 
Let $g\in\R[Z_1,\dots,Z_n]$ such that $f = g(e_1,\dots, e_n)$. There is $x_0\in S$ with
\[ \inf_{x\in S} f(x)=f(x_0) \]
consider $z_0:=\Gamma(x_0)$ and $a:=L(z)$. Since $g$ is concave and not constant on $\H_L(a)$, $g$ attains its minimum on an extreme point of $\H_L(a)$, i.e., we can assume that $z_0\in \H_L^k(a)$ and therefore $x_0\in \A_k$.
\end{proof}

The existence of local extreme points in Theorem \ref{thm:4} (1) is indeed necessary, as in cases without local extreme points it is possible to construct situations where the statement will not hold. We showcase this in the following.
\begin{example}\label{example4}
Let $K(h)=\R^4$, $l_1:=Z_2$, $l_2:=Z_4$ and $L: ~\R^4\to \R^2, ~ z\mapsto (l_1(z),l_2(z))$ and consider the $(l_1,l_2)$-sufficient symmetric polynomial
\[f= (e_2+5)^2+(e_4-4)^2 \in \R[X_1,X_2,X_3,X_4].\]
The $2$-boundary $\H^2_L(-5,4)$ is empty by Example \ref{example3} (3). So $f(x)> 0$ for all $x\in \A_2$, but $f(1,-1,2,-2)=0$.
\end{example}

One can in fact prove that the polynomial $f$ in Example \ref{example4} is still $3$-complete. Indeed, the necessity of the existence of an extreme point in every hyperbolic slice seems to restrict the applications of Theorem \ref{thm:4}. However, by applying  Lemma \ref{lemmalocalextremepoint} and Lemma \ref{compact} we can obtain the following version of Theorem \ref{thm:4} which avoids this issue at the price of a slightly weaker conclusion.

\begin{corollary}
\label{cor:thm:4}
Let $S\subseteq \R^n$ be a symmetric $(l_1,\dots,l_k)$-sufficient semi-algebraic set and let $f\in\R[\X]$ be a symmetric polynomial.
\begin{enumerate}
 \item If $f$ is $(l_1,\dots,l_k)$-sufficient, then $f$ is $(k+1)$-complete on $S$.
 \item If $f$ is $(l_1,\dots,l_k)$-concave-sufficient, then $f$ is $(k+2)$-testable on $S$.
 \item If $f\in\R[\X]$ is $(l_1,\dots,l_k)$-quasi-sufficient and $S\cap \A_k$ is connected, then $f$ is $(k+2)$-complete on $S$.
\end{enumerate}
 Moreover if $Z_1\in \Span(l_1,\dots,l_k)$, then $(k+1)$ in $(1)$ can be replaced by $k$-complete. If $Z_1,Z_2\in \Span(l_1,\dots,l_k)$, then $(k+2)$ in $(2)$ and $(3)$ can be replaced by $k$.
\end{corollary}

The results in this section were given entirely for symmetric functions. To conclude this section we remark the following direct translation of the results to even symmetric polynomials or equivalently copositive symmetric polynomials.
\begin{remark}\label{evensymmetric}
The results on symmetric polynomials translate directly to even symmetric polynomials, i.e., polynomials invariant by the natural action of the Hyperoctahedral group $S_2\wreath S_n$. Denote by
\begin{align*}
\E :&= \{z\in \R^n ~|~ T^{2n}-z_1T^{2(n-1)}+\dots+(-1)^n z_n \text{ is hyperbolic}\}\\
&=\{z\in \H ~|~ T^{n}-z_1T^{n-1}+\dots+(-1)^n z_n \text{ has only non-negative roots} \}\end{align*}
the set of even hyperbolic polynomials. Furthermore, we define
\[\E^k:=\{z\in \E ~|~ T^{n}-z_1T^{n-1}+\dots+(-1)^n z_n \text{ has at most $k$ positive roots} \} \]
and $\E_L(a):=\E \cap L^{-1}(a)$ and $\E_L^k(a)$ accordingly. Then the proof of Theorem \ref{maintheorem}
translates to $\locextr(\E_L(a))\subseteq \E^k_L(a)$ and both sets are generically finite. By replacing $\A_k$ by \[\B_k:=\{ x\in\R^n ~:~ |\{x_1^2,\dots,x_n^2\}\setminus \{0\}|\leq k \}\] we can transfer the statements of Theorem \ref{thm:4} and Corollary \ref{cor:thm:4} about $k$-completeness and $k$-testability of (quasi-)sufficient symmetric polynomials to (quasi-)sufficient even symmetric polynomials $f$, i.e., polynomials that admit a representation of the form \[f=g(e_1(X_1^2,\dots,X_n^2),\dots,e_n(X_1^2,\dots,X_n^2))\] with $g\in \R[l_1,\dots,l_k]$. Note that in this case it suffices already to fix the first coefficient in order to obtain compactness, so one can replace $(k+2)$ in Corollary \ref{cor:thm:4} $(2)$ and $(3)$ by $(k+1)$. 
\end{remark}

\subsection{Deciding sufficiency}
Generally the definition of sufficient and quasi-sufficient given above can appear to be not  directly verifiable. Especially since mostly one is given a symmetric polynomial without its representation in terms of linear combinations of elementary symmetric polynomials. Therefore, we want to shortly present how to algorithmically approach the question if a given symmetric polynomial is sufficient or quasi-sufficient. In order to decide if a symmetric polynomial $f \in \R[\X]$ is sufficient for some collection of linear forms $l_1,\ldots,l_k$ one has principle two task:

\begin{enumerate}
 \item Finding a representation of $f=g(e_1,\ldots,e_n)$ in terms of elementary symmetric polynomials: This can be achieved, for example, by using the Gröbner basis 
 $G:=\{g_1,\dots,g_k\}$, where 
 \[g_k=\sum_{\substack{\alpha\in \N_0^{n-k+1}\\|\alpha|=k}}X_k^{\alpha_1}\cdots X_n^{\alpha_{n-k+1}} 
 +\sum_{i=1}^k (-1)^iY_i 
 \sum_{\substack{\alpha\in \N_0^{n-k+1}\\|\alpha|=k-i}}X_k^{\alpha_1}\cdots X_n^{\alpha_{n-k+1}}\]
 of the ideal $I=(e_1-Y_1,\dots,e_n-Y_n) \subseteq \R[\X,Y_1,\dots,Y_n]$ which is independent from $f$ and then by computing the remainder $g$ of $f$ on division by $G$. One obtains now $f=g(e_1,\dots,e_n)$ (see Proposition 4 and Proposition 5 in §1 of Chapter 7 in \cite{cox2013ideals} for details).
 Alternatively one can use the algorithm presented in \cite{vu2022complexity}. 
 \item Once $g\in\R[e_1,\ldots,e_n]$ is obtained, one has to decide if there exist $k<n$ linear combinations $l_1,\ldots,l_k$ of the $e_1,\ldots,e_n$ such that $g\in \R[l_1,\dots,l_k]$. Also this can be accomplished quite concretely, for example, by using the approach outlined by Carlini \cite{carlini2006reducing}. As described there, the smallest number $k$ of linear forms $l_1,\dots,l_k$ needed such that $g\in \R[l_1,\dots,l_k]$ is obtained by computing the rank of the Catalectican matrix of $g$. This matrix is obtained by the coefficients of the partial derivatives of $g$. More concretely, one can actually also explicitly construct these linear forms by computing a basis for the vector space of the $(d-1)$-th partial derivatives of $g$. 
\end{enumerate}

\noindent The steps described above rely mostly on linear algebra and can be efficiently implemented also for larger numbers of variables.

\begin{remark}
In the special case when one wants to decide if a symmetric polynomials $f$ is $e_{i_1},\ldots,e_{i_m}$-quasi-sufficient (where $1\leq i_1 \leq \dots \leq i_n \leq k$) one can actually proceed with the following examination of the gradient of $f$ without going through the steps above:
 As a symmetric polynomial $f$ cane be written as $f=g(e_1,\dots,e_n)$ we have 
 \[\nabla f=\nabla g J_{e_1,\dots,e_n}.\]
 Noting that $J_{e_1,\dots,e_n}$ is invertible over $\R(X_1,\ldots,X_n)$ we get 
 \[\nabla f J_{e_1,\dots,e_n}^{-1} =\nabla g.\]
 Now, if for $I\subseteq \{1,\dots,n\}$ the corresponding entries in $\nabla g$ are constants, then $f$ is $(e_i)_{\{1,\dots,n\} \setminus I}$-quasi-sufficient.
\end{remark}
We give a short example to illustrate the algorithmic approach. 
\begin{example}
We consider the following toy example of a symmetric polynomial in three variables in order to showcase the methods described above
\begin{align*}
f=&\sum_{\sigma \in S_3}\sigma \,\Bigg(\frac{1}{2}\,X_1^3
+X_1^2X_2^2+3X_1^2X_2+X_1^3X_2+X_1X_2X_3-X_1^2X_2^2X_3^2\\
&+\frac{1}{2}X_1^3X_2^3X_3^2-2 X_1^3X_2^2X_3-X_1^3X_2X_3-2X_1^2X_2^2X_3+\frac{5}{2}X_1^2X_2X_3\Bigg),
\end{align*}
where $S_3$ acts on $\R[X_1,X_2,X_3]$ by permutation of variables.

The Gr\"obner basis corresponding to the ideal 
\[I:=\langle e_1-Y_1,e_2-Y_2,e_3-Y_3 \rangle\] is given by 
\[G=\{X_1 + X_2 + X_3 - Y_1, X_2^2 + X_2X_3 - X_2Y_1 + X_3^2 - X_3Y_1 + Y_2, X_3^3 - X_3^2Y_1 + X_3Y_2 - Y_3\}.\]
By computing the remainder of $f$ on division by $G$ one obtains
\[g=Y_1^3 + Y_1^2Y_2 - 2Y_1^2Y_3 - 2Y_1Y_2Y_3 + Y_1Y_3^2 + Y_2Y_3^2\in \R[Y_1,Y_2,Y_3]\]
with $f=g(e_1,e_2,e_3)$. In order to compute the Catalactican of $g$, we fix a monomial basis \[M=\{M_1,\dots,M_6\}=\{Y_1^2,Y_1Y_2,Y_1Y_3,Y_2^2,Y_2Y_3,Y_3^2\}\] 
for the ternary forms of degree $2=\deg(g)-1$. Calculating the partial derivatives 
\[\partial_i g = c_{i1}M_1+\dots+c_{i6}M_6 \] we obtain 
he Catalactican $C_g$ of $g$ defined as $(C_g)_{ij}=c_{ij},$ i.e.
\[C_g=\begin{pmatrix}
 3 & 2 & -4 & 0 & -2 & 1 \\
1 & 0 & -2 & 0 & 0 & 1 \\
-2 & -2 & 2 & 0 & 2 & 0
\end{pmatrix}.\]
The number of linear forms needed to express $g$ is then equal to $\rank(C_g)=2$. In order to find linear forms needed to express $g$, it suffices to compute a basis for the span of the second partial derivatives of $g$, we obtain
\[\{Y_1-Y_3,Y_2+Y_3\}\]
and indeed \[g=(Y_2 + Y_3)(Y_1 - Y_3)^2 + (Y_1 - Y_3)^3,\] i.e.
$f$ is $(Y_2 + Y_3,Y_1 - Y_3)$-sufficient and $(Y_1-Y_3)$-quasi-sufficient.
\end{example}

\section{Applications and examples}\label{sec:ex}
We will now show some applications of the theory developed here and use it on some concrete examples to underline the potential of the results presented. We begin with examining the following polynomial which was given by Robinson \cite{robinson1969some} as an example of a non-negative form which is not a sum of squares. Note that this example could also be  obtained by a variant of the half degree principle to even symmetric polynomials.  
\begin{example}[Robinson Polynomial]
The non-negativity of the Robinson polynomial
\[R=X^6+Y^6+Z^6-\left(X^4Y^2+X^2Y^4+X^4Z^2+X^2Z^4+Y^4Z^2+Y^2Z^4\right)+3X^2Y^2Z^2\]
 can be easily verified using Remark \ref{evensymmetric}. Indeed,
\[R=e_1(X^2,Y^2,Z^2)^3-4e_1(X^2,Y^2,Z^2)e_2(X^2,Y^2,Z^2)+9e_3(X^2,Y^2,Z^2)\]
is a $Z_1$-quasi-sufficient even symmetric polynomial. Therefore, we only need to examine $R$ on the set \[\B_1:=\{ x\in\R^3 \,:\, |\{x_1^2,x_2^2,x_3^2\}\setminus \{0\}|\leq 1 \}.\]
Since we easily find that the two (dehomogenized) univariate polynomials
\begin{align*}
 R_{1}&=R(1,T,T)=T^4 - 2T^2 + 1=(T - 1)^2(T + 1)^2\\
 R_2&=R(1,T,0)=T^6 - T^4 - T^2 + 1=(T^2 + 1)(T - 1)^2(T + 1)^2.
\end{align*}
are non-negative, $R$ is indeed non-negative. 
Moreover, we directly also see that $R$ has at least the $10$ projective zeros 
\[ (1,\pm 1,\pm 1),(0,\pm 1, \pm 1), (\pm 1,0, \pm 1), (\pm 1, \pm 1,0)\] which constitute the orbits of $(1,1,1)$ and $(1,1,0)$. One easily checks that these zeros are isolated. From this observation one immediately also obtains that $R$ cannot be a sum of squares. Indeed, since a zero of a sum of squares also has to be a zero of every summand, a sextic which is a sum of squares can have at most $9$ isolated zeros.
\end{example}

Furthermore, we will show how our results can be used to verify symmetric inequalities rather easily. 
\begin{example}[AM–GM inequality]
The inequality of arithmetic and geometric means is a standard inequality from analysis, stating that for all $x\in\R_{\geq 0}^n$ we have
\[\frac{x_1 + x_2 + \cdots + x_n}{n} \ge \sqrt[n]{x_1 \cdot x_2 \cdots x_n},\]
or equivalently 
\[e_1^n-n^n e_n \geq 0 ~ \text{on} ~ \R_{\geq 0}^n.\]
By squaring the variables this is equivalent to 
\[F=e_1(X_1^2,\dots,X_n^2)^n-n^n e_n(X_1^2,\dots,X_n^2)\]
is non-negative, which can be proven by applying again Remark \ref{evensymmetric} similarly to the previous example.
\end{example}

\begin{example}[Maclaurin's inequality]
More general we have
\[\sqrt[i]{\frac{e_i(x)}{\binom{n}{i}}}\geq \sqrt[j]{\frac{e_j(x)}{\binom{n}{j}}}\]
for all $x\in \R_{\geq 0}^n$ and $i\leq j$ which is equivalent to
\[F=\binom{n}{j}^{2i} e_i(X_1^2,\dots,X_n^2)^{2j} - \binom{n}{i}^{2j} e_j(X_1^2,\dots,X_n^2)^{2i}\]
is non-negative. $F$ is $(Z_i)$-concave-sufficient and even symmetric. First we show that $\inf_{x\in \R^k} f > -\infty$. Since $F$ is in particular $(Z_1,Z_i)$-concave-sufficient, it suffices to show that
\[
F_\lambda:=F(\underbrace{X,\dots,X}_{\lambda_1-\text{times}},\underbrace{Y,\dots,Y}_{\lambda_2-\text{times}},\underbrace{0,\dots,0}_{\lambda_3-\text{times}})
\]
is bounded from below for all partitions $\lambda_1+\lambda_2+\lambda_3=n$. Since $F_\lambda$ is homogeneous it suffices to show that the dehomogenization
\[\tilde{F_\lambda}=F_\lambda(X,1)\]
has positive leading coefficient. It has leading coefficient
\[\binom{n}{j}^{2i} \binom{\lambda_1}{i}^{2j} - \binom{n}{i}^{2j} \binom{\lambda_1}{j}^{2i}> 0\]
for $i\leq\lambda_1<n$ (this can be easily shown by induction on $\lambda_1$) and $\tilde{F_\lambda}=0$ for $\lambda_1=n$ and for $\lambda_1< i$. Now we can use Theorem \ref{thm:4} $(4)$ and Remark \ref{evensymmetric}, so it suffices to check that
\[
F_\mu:=F(\underbrace{X,\dots,X}_{\mu-\text{times}},\underbrace{0,\dots,0}_{(n-\mu)-\text{times}})
\]
is non-negative for all partitions $\mu+n-\mu=n$. Since $F_\mu$ is homogeneous it suffices to show that the dehomogenization
\[\tilde{F_\mu}=F_\mu(1)=\begin{cases}
\binom{n}{j}^{2i} \binom{\mu}{i}^{2j} - \binom{n}{i}^{2j} \binom{\mu}{j}^{2i}, & \text{for } i\leq \mu < n\\
0, & \text{else}
\end{cases}\]
is non-negative.
\end{example}

It is interesting to notice that the idea of certifying symmetric inequalities in the way sketched has been done albeit not as general. For example, the main Lemma \cite[Lemma 2.4]{mitev2003new} used to prove some new inequalities between elementary symmetric polynomials can be seen as a special case of Remark \ref{evensymmetric} for $Z_1$-quasi-sufficient even symmetric polynomials. To close we remark that, indeed, our setup also recovers as a special instance of Corollary \ref{cor:thm:4} together with Proposition \ref{prop:(half-)degree} the so called Degree and Half-Degree Principle shown in \cite{timofte2003positivity}.

\begin{corollary}[Degree Principle]
Let $S\subseteq \R^n$ be a symmetric semi-algebraic set, which can be described by symmetric polynomials of degree at most $d$. Then $S$ is empty, if and only if $S \cap \A_d$ is empty.
\end{corollary}

\begin{corollary}[Half-Degree Principle]
Let $f \in \R[\X]$ be symmetric of degree $d$. Then $f$ is $k$-complete, where $k:=\max\left\{2,\floor*{\frac{d}{2}}\right\}$.
\end{corollary}

We remark that it is known to be NP-hard already for quartics to decide non-negativity (see e.g. \cite{blum1998complexity} or \cite{murty1985some}). However, for univariate polynomials non-negativity can be certified via a sums of squares decomposition. Such a decomposition can be efficiently obtained via semi-definite programming. The feasible region of a semi-definite program is given by a linear matrix inequality (LMI), i.e., an inequality of the form $A_0+x_1A_1+x_2A_2+\ldots+x_nA_n\succeq 0$,
where $A_0,\ldots, A_{n}$ are real symmetric matrices all of the same size and $x_1,\ldots,x_n$ are supposed to be real scalars. Now for a symmetric $1$-complete polynomial of degree $2d$ we have that $f$ is non-negative if and only if the univariate polynomial $\tilde f:=f(T,T,\ldots,T)$ of same degree is non-negative. This in turn is the case, if and only if there exists a symmetric matrix $A\in \R^{(d+1)\times (d+1)}$ which is non-negative and for which we have $\tilde{f}=(1,T,T^2,\ldots,T^n) \cdot A\cdot(1,T,T^2,\ldots,T^n)^t$. Therefore, non-negativity of a $1$-complete symmetric polynomial can be decided with semi-definite programming. 
This motivates the following sufficient criterion for $1$-complete polynomials.
\begin{theorem}\label{thm:1suff}
Let $l\in \R[Z_1,\dots,Z_n]_1$ be linear and homogeneous, say $l=\lambda_1 Z_1+\cdots+\lambda_n Z_n$ for some $\lambda_1,\dots,\lambda_n\in \R$. Let $f$ be a $l$-sufficient symmetric polynomial. Let $m$ denote the largest index $i$ of the non-zero $\lambda_i$, i.e., $m:=\max\left\{i\in\{1,\dots,n\} ~|~ \lambda_i\neq 0 \right\}$. If $m$ is odd, then $f$ is $1$-complete.
\end{theorem}
\begin{proof}
Write $f$ as $f:=g(l(e_1\dots,e_n))$ for some univariate polynomial $g$. Let $x\in \R^n$ and define $a:=l(e_1(x),\dots,e_n(x)) \in \R$. We will show that $\H_l^1(a)\neq \emptyset$. Consider the univariate polynomial
\[p:= \sum_{i=1}^m \lambda_i \binom{n}{i} T^i - a \in \R[T]. \]
Since $m$ is odd, $p$ has a real zero $y\in \R$. Consider now $z=(z_1,\dots,z_n)\in \R^n$ defined by $z_i:=\binom{n}{i} y^i$. Then $z \in \H_l^1(a)$ by construction. Now
\[f(x)=g(a)=g(l(z_1,\dots,z_n))=f(y,\dots,y).\]
\end{proof}

Convex sets for which membership can be described via semi-definite programming, i.e., which are projections of feasibility regions of semi-definite programs are called \emph{spectrahedral shadows}. Recently, Scheiderer \cite{scheiderer2018spectrahedral} was able to show that in general the cone of positive semi-definite forms is not in general a spectrahedral shadow. Using Corollary \ref{cor:thm:4} and Remark \ref{evensymmetric} we can identify families of convex cones of (even-)symmetric positive semi-definite forms which are spectrahedral shadows, generalizing Theorem 4.29 in \cite{debus2020reflection}.

\begin{proposition}\label{cor:spectrahedral}
    Let $\P_{2d}$ denote the convex cone of positive semi-definite $n$-ary forms of degree $2d$ and $2\leq j\leq n$. Then, the subcones of all $(Z_1,Z_j)$-sufficient and $(Z_1,Z_2)$-quasi-sufficient symmetric forms are spectrahedral shadows. Similarly, the subcone of all $(Z_1,Z_j)$-quasi-sufficient even-symmetric forms is a spectrahedral shadow.
\end{proposition}
\begin{proof}
    All forms in the mentioned subcones are $2$-complete by Corollary \ref{cor:thm:4} and Remark \ref{evensymmetric}. Therefore non-negativity can be decided by restricting to $\A_2$, respectively $\B_2$. Dehomogenizing the resulting binary forms we obtain univariate polynomials, which are non-negative if and only if they are sums of squares.
\end{proof}

\section{Conclusion and open questions}\label{sec:con}
We have defined the notion of hyperbolic slices and showed that the local extreme points of such slices correspond to hyperbolic polynomials with few distinct roots. We show that generically these hyperbolic slices contain at most finitely many local extreme points. We expect that this holds generally, i.e., also in those cases when the $k$-boundary is not finite. In particular, we expect that the convex hull of each connected component of any hyperbolic slice is a polyhedron. Arnold and Giventhal \cite{arnol1986hyperbolic,givental1987moments} had shown that the hyperbolic slices which are obtained by fixing the first $k$ coefficients are contractible. Our examples show that hyperbolic slices are in general neither connected nor compact and therefore in particular not contractible. It would be very interesting to study the topological properties of these sets. Similarly to the results in \cite{basu}, an understanding of the topology of these slices might allow for new efficient algorithms to compute the homology of symmetric semi-algebraic sets defined by $k$-complete polynomials. 
Furthermore, the definition of hyperbolic slices naturally involved elementary symmetric polynomials. From the viewpoint of symmetric polynomials, it seems interesting to study analogous sets for different choices of $n$ symmetric polynomials which generate all symmetric polynomials. For example, the first author observed in \cite{riener2016symmetric} that symmetric polynomials defined by any $k$ Newton sums are at least $(2k+1)$-complete. Finally, a natural question is to explore the connections to invariant polynomials of other groups, most notably finite reflection groups. In \cite{raman,velasco} the authors showed that the image of polynomial functions invariant by a finite reflection group can be described by the points on flats in the hyperplane arrangement, if the degree is sufficiently small. We expect that the notions and techniques presented here can be transferred also to this more general setup.

\section*{Acknowledgments}
This work has been supported by the Tromsø Research Foundation (grant agreement 17matteCR). The authors would like to thank Philippe Moustrou for his valuable comments on the manuscript as well as an anonymous referee whose suggestions and remarks on a previous version of this article gave important impulses. 
\bibliographystyle{abbrv}
\bibliography{references}

\begin{thebibliography}{10}

\bibitem{velasco}
J.~Acevedo and M.~Velasco.
\newblock Test sets for nonnegativity of polynomials invariant under a finite
  reflection group.
\newblock {\em J. Pure Appl. Algebra}, 220(8):2936--2947, 2016.

\bibitem{arnol1986hyperbolic}
V.~I. Arnol'd.
\newblock Hyperbolic polynomials and vandermonde mappings.
\newblock {\em Funktsional'nyi Analiz i ego Prilozheniya}, 20(2):52--53, 1986.

\bibitem{Basu2003}
S.~Basu, R.~Pollack, and M.-F. Roy.
\newblock {\em Algorithms in Real Algebraic Geometry}.
\newblock Springer, Berlin, Heidelberg, 2003.

\bibitem{basu}
S.~Basu and C.~Riener.
\newblock Vandermonde varieties, mirrored spaces, and the cohomology of
  symmetric semi-algebraic sets.
\newblock {\em Foundations of Computational Mathematics}, 2021.

\bibitem{blum1998complexity}
L.~Blum, F.~Cucker, M.~Shub, and S.~Smale.
\newblock {\em Complexity and real computation}.
\newblock Springer Science \& Business Media, 1998.

\bibitem{branden}
P.~Br{\"a}nd{\'e}n.
\newblock Obstructions to determinantal representability.
\newblock {\em Advances in Mathematics}, 226(2):1202--1212, 2011.

\bibitem{carlini2006reducing}
E.~Carlini.
\newblock Reducing the number of variables of a polynomial.
\newblock In {\em Algebraic geometry and geometric modeling}, pages 237--247.
  Springer, 2006.

\bibitem{cox2013ideals}
D.~Cox, J.~Little, and D.~OShea.
\newblock {\em Ideals, varieties, and algorithms: an introduction to
  computational algebraic geometry and commutative algebra}.
\newblock Springer Science \& Business Media, 2013.

\bibitem{debus2020reflection}
S.~Debus and C.~Riener.
\newblock Reflection groups and cones of sums of squares.
\newblock {\em arXiv preprint arXiv:2011.09997}, 2020.

\bibitem{dedieu1992obreschkoff}
J.~P. Dedieu.
\newblock Obreschkoff's theorem revisited: what convex sets are contained in
  the set of hyperbolic polynomials?
\newblock {\em Journal of pure and applied algebra}, 81(3):269--278, 1992.

\bibitem{faugere2023computing}
J.-C. Faug{\`e}re, G.~Labahn, M.~S. El~Din, {\'E}.~Schost, and T.~X. Vu.
\newblock Computing critical points for invariant algebraic systems.
\newblock {\em Journal of Symbolic Computation}, 116:365--399, 2023.

\bibitem{foregger1987relative}
T.~H. Foregger.
\newblock On the relative extrema of a linear combination of elementary
  symmetric functions.
\newblock {\em Linear and Multilinear Algebra}, 20(4):377--385, 1987.

\bibitem{raman}
T.~Friedl, C.~Riener, and R.~Sanyal.
\newblock Reflection groups, reflection arrangements, and invariant real
  varieties.
\newblock {\em Proceedings of the American Mathematical Society},
  146(3):1031--1045, 2018.

\bibitem{givental1987moments}
A.~B. Givental.
\newblock Moments of random variables and the equivariant morse lemma.
\newblock {\em Russian Mathematical Surveys}, 42(2):275--276, 1987.

\bibitem{guler1997hyperbolic}
O.~G{\"u}ler.
\newblock Hyperbolic polynomials and interior point methods for convex
  programming.
\newblock {\em Mathematics of Operations Research}, 22(2):350--377, 1997.

\bibitem{gurvits2006hyperbolic}
L.~Gurvits.
\newblock Hyperbolic polynomials approach to van der waerden/schrijver-valiant
  like conjectures: sharper bounds, simpler proofs and algorithmic
  applications.
\newblock In {\em Proceedings of the thirty-eighth annual ACM symposium on
  Theory of computing}, pages 417--426, 2006.

\bibitem{keilson1967global}
J.~Keilson.
\newblock On global extrema for a class of symmetric functions.
\newblock {\em Journal of Mathematical Analysis and Applications},
  18(2):218--228, 1967.

\bibitem{kostov1989geometric}
V.~Kostov.
\newblock On the geometric properties of vandermonde's mapping and on the
  problem of moments.
\newblock {\em Proceedings of the Royal Society of Edinburgh Section A:
  Mathematics}, 112(3-4):203--211, 1989.

\bibitem{kostov2007hyperbolic}
V.~P. Kostov.
\newblock On hyperbolic polynomial-like functions and their derivatives.
\newblock {\em Proceedings of the Royal Society of Edinburgh Section A:
  Mathematics}, 137(4):819--845, 2007.

\bibitem{kostov2002arrangements}
V.~P. Kostov and B.~Z. Shapiro.
\newblock On arrangements of roots for a real hyperbolic polynomial and its
  derivatives.
\newblock {\em Bulletin des sciences mathematiques}, 126(1):45--60, 2002.

\bibitem{kovavcec2012note}
A.~Kova{\v{c}}ec, S.~Kuhlmann, and C.~Riener.
\newblock A note on extrema of linear combinations of elementary symmetric
  functions.
\newblock {\em Linear and Multilinear Algebra}, 60(2):219--224, 2012.

\bibitem{lombardi2014elementary}
H.~Lombardi, D.~Perrucci, and M.-F. Roy.
\newblock An elementary recursive bound for effective positivstellensatz and
  hilbert 17-th problem.
\newblock {\em arXiv preprint arXiv:1404.2338}, 2014.

\bibitem{marcus}
A.~W. Marcus, D.~A. Spielman, and N.~Srivastava.
\newblock Interlacing families ii: Mixed characteristic polynomials and the
  kadison—singer problem.
\newblock {\em Annals of Mathematics}, 182(1):327--350, 2015.

\bibitem{meguerditchian1991geometrie}
I.~Meguerditchian.
\newblock {\em G{\'e}om{\'e}trie du discriminant r{\'e}el et des polyn{\^o}mes
  hyperboliques}.
\newblock PhD thesis, Rennes 1, 1991.

\bibitem{meguerditchian1992theorem}
I.~Meguerditchian.
\newblock A theorem on the escape from the space of hyperbolic polynomials.
\newblock {\em Mathematische Zeitschrift}, 211(1):449--460, 1992.

\bibitem{mitev2003new}
T.~P. Mitev.
\newblock New inequalities between elementary symmetric polynomials.
\newblock {\em Journal of Inequalities in Pure and Applied Mathematics},
  4(2):2003, 2003.

\bibitem{Mondal2021}
P.~Mondal.
\newblock {\em Number of zeroes on the affine space I: (Weighted) B{\'e}zout
  theorems}, pages 207--214.
\newblock Springer International Publishing, Cham, 2021.

\bibitem{murty1985some}
K.~G. Murty and S.~N. Kabadi.
\newblock Some {NP}-complete problems in quadratic and nonlinear programming.
\newblock {\em Math. Programming}, 39(2):117--129, 1987.

\bibitem{pfister1967darstellung}
A.~Pfister.
\newblock Zur {D}arstellung definiter {F}unktionen als {S}umme von {Q}uadraten.
\newblock {\em Inventiones mathematicae}, 4(4):229--237, 1967.

\bibitem{riener2012degree}
C.~Riener.
\newblock On the degree and half-degree principle for symmetric polynomials.
\newblock {\em Journal of Pure and Applied Algebra}, 216(4):850--856, 2012.

\bibitem{riener2016symmetric}
C.~Riener.
\newblock Symmetric semi-algebraic sets and non-negativity of symmetric
  polynomials.
\newblock {\em Journal of Pure and Applied Algebra}, 220(8):2809--2815, 2016.

\bibitem{robinson1969some}
R.~M. Robinson.
\newblock Some definite polynomials which are not sums of squares of real
  polynomials.
\newblock In {\em Notices of the American Mathematical Society}, volume~16,
  page 554, 1969.

\bibitem{scheiderer2018spectrahedral}
C.~Scheiderer.
\newblock Spectrahedral shadows.
\newblock {\em SIAM Journal on Applied Algebra and Geometry}, 2(1):26--44,
  2018.

\bibitem{timofte2003positivity}
V.~Timofte.
\newblock On the positivity of symmetric polynomial functions.: Part i: General
  results.
\newblock {\em Journal of Mathematical Analysis and Applications},
  284(1):174--190, 2003.

\bibitem{vu2022complexity}
T.~X. Vu.
\newblock On the complexity of invariant polynomials under the action of finite
  reflection groups.
\newblock {\em arXiv preprint arXiv:2203.04123}, 2022.

\bibitem{waterhouse1983symmetric}
W.~C. Waterhouse.
\newblock Do symmetric problems have symmetric solutions?
\newblock {\em The American Mathematical Monthly}, 90(6):378--387, 1983.

\bibitem{Whitney}
H.~Whitney.
\newblock {\em Complex analytic varieties}.
\newblock Addison-Wesley Publishing Co., Reading, Mass.-London-Don Mills, Ont.,
  1972.

\end{thebibliography}
\end{document}